\documentclass{article}
\usepackage{graphicx} 

\usepackage{amssymb, enumerate}
\usepackage{mathtools}
\usepackage{amsmath}
\usepackage{amsthm}
\usepackage{hyperref}
\usepackage{enumitem}    
\usepackage[capitalise]{cleveref}
\usepackage{xcolor}
\usepackage{tikz}
\usetikzlibrary{calc}
\usepackage{geometry}

 \date{}

\newtheorem{proposition}{Proposition}
\newtheorem{theorem}{Theorem}
\newtheorem{corollary}{Corollary}
\newtheorem{lemma}{Lemma}
\newtheorem{problem}{Problem}
\newtheorem{conjecture}{Conjecture}
\newtheorem*{conjecture*}{Conjecture}
\newtheorem*{fact*}{Fact}

\renewcommand{\pm}{\Phi}

\providecommand{\keywords}[1]
{
  {\small	
  \textbf{\textit{Keywords---}} #1}
}

\providecommand{\msc}[1]
{
  {\small	
  \textbf{\textit{Math. Subj. Class. (2020)---}} #1}
}

\title{On the number of perfect matchings in planar graphs}

\begin{document}

\author{Jan Goedgebeur\inst{1,2}
\and
Jorik Jooken\inst{1}
\and
Michiel Provoost\inst{1}
\and
Carol T. Zamfirescu\inst{2,3}
}

\author{
Jan Goedgebeur \thanks{Department of Computer Science, KU Leuven Campus Kulak-Kortrijk, Kortrijk, Belgium.} 
 \thanks{Department of Mathematics, Computer Science and Statistics, Ghent University, Ghent, Belgium.}
\and Jorik Jooken
 \footnotemark[1]
\and Tibo Van den Eede
\footnotemark[1]
\thanks{School of Computing, Australian National University, Canberra, Australia.}
\and Carol T. Zamfirescu
 \thanks{Centre for Research in Mathematics and Data Science, Western Sydney University, Australia.\\ Email addresses:
 \protect\href{mailto:jan.goedgebeur@kuleuven.be}{\protect\nolinkurl{jan.goedgebeur@kuleuven.be}},
\protect\href{mailto:jorik.jooken@kuleuven.be}{\protect\nolinkurl{jorik.jooken@kuleuven.be}},
\protect\href{mailto:tibo.vandeneede@kuleuven.be}{\protect\nolinkurl{tibo.vandeneede@kuleuven.be}} and
\protect\href{mailto:czamfirescu@gmail.com}{\protect\nolinkurl{czamfirescu@gmail.com}}
}
\footnotemark[2]
}

\maketitle

\begin{abstract}
    We investigate the minimum non-zero number of perfect matchings in planar graphs. We prove that this is a constant for 2-connected planar graphs of minimum degree 3 and 3-connected planar graphs. In the former case, the constant is 4 and this is best possible. In the 3-connected case, we give several infinite families, including nearly 3-regular graphs and triangulations with a constant number of perfect matchings. In contrast, it was known that in the 4-connected case the minimum non-zero number of perfect matchings is at least linear. For 5-connected triangulations, it follows from a result of Alahmadi, Aldred, and Thomassen that there must be exponentially many perfect matchings. The families of planar graphs we investigate here are classified by connectivity. We conclude the article with an infinite family of counterexamples to a conjecture published by Zaks; these are not planar, but very much concern connectivity constraints.
\end{abstract}


\vspace*{1pt}

\msc{05C70, 05C10, 05C40, 05C35}

\section{Introduction}

Chudnovsky and Seymour~\cite{CS12} proved that planar cubic bridgeless graphs have at least an exponential number of perfect matchings; this was generalised by Esperet, Kardo{\v{s}}, King, Kr{\'a}l' and Norine~\cite{Esperet2011} to not necessarily planar graphs. Motivated by the former paper, we here drop the regularity condition and investigate how the number of perfect matchings behaves in planar graphs. Two surprising observations motivated the present work. 
Firstly, there exist infinitely many planar 2-connected $n$-vertex minimum degree~3 graphs with $3n/2+1$ edges and a positive constant number of perfect matchings. In this paper we present such an infinite family and several related infinite families of planar graphs with a positive constant number of perfect matchings. Secondly, while Zaks~\cite{Z71} described $k$-connected graphs with $k!$ perfect matchings, we shall see that in planar graphs this behaviour is drastically different; in particular, there is a stark shift between the 3-connected case and the 4-connected case.

In \cref{sec:2and3connPl} we discuss planar graphs that are 2- or 3-connected. We present infinite families of such graphs with a constant number of perfect matchings; for the 2-connected case, our construction is optimal. 

In \cref{sec:4and5connPl} we treat planar 4- and 5-connected  graphs. It is known that these have at least a linear number of perfect matchings, but there exist planar 4-connected graphs with a quadratic number of perfect matchings. 

This section is mostly a survey of consequences of existing results. The families of planar graphs we present in the aforementioned results are classified by connectivity. In \cref{sec:conjZaks} we describe a structurally simple infinite family of counterexamples to a conjecture published by Zaks~\cite{Z71}, but perhaps not due to him but Gr\"unbaum (Zaks himself does not make this clear in~\cite{Z71}); these counterexamples are not planar, but do very much concern connectedness. 
We conclude with \cref{sec:notes} in which we make some comments on the work presented here, and present open problems and a conjecture.

Furthermore, the graphs appearing in figures in this paper are made available on both \textit{the House of Graphs~\cite{HoG}} and the GitHub repository \href{https://github.com/AGT-Kulak/countpm}{https://github.com/AGT-Kulak/countpm}. These graphs can be found on the House of Graphs by searching with the keyword \texttt{perfmatchplanar}. In addition, the GitHub repository contains code to count the number of perfect matchings of graphs, which can be used to verify the perfect matchings counts in this paper and might be of use to other researchers studying the number of perfect matchings in graphs.

\subsection{Preliminaries and definitions}

We now introduce our notation and give a useful auxiliary result. All graphs considered in this paper are simple. For a graph $G$, let $\pm(G)$ denote the number of perfect matchings in $G$.  
    
    We will call a graph \textit{matchable} if it contains a perfect matching, and a family of graphs is \textit{matchable} if each of its members is matchable.

A planar graph is a \textit{triangulation} if it has at least four vertices and the addition of any edge renders it non-planar. 
 We note that a triangulation here thus is 3-connected. A graph is \textit{outerplanar} if it has a planar embedding in which all vertices lie on the outer face.

For a given graph $G$, like Zaks~\cite{Z71} we let $F(G)$ denote the subgraph of $G$, which is the union of all of the 1-factors of $G$ (note that Mader~\cite{M76} writes $F(G)$ for $\pm(G)$). 

When $G$ is connected, has at least two vertices, and $F(G) = G$, the graph $G$ is called \textit{matching covered}, since all of its edges are contained in a perfect matching.
An edge or subset of edges in $E(F(G))$ or $E(G) \setminus E(F(G))$ is called \textit{admissible} or \textit{inadmissible}, respectively. 
Given a graph $G$, let $D \subseteq V(G)$, and let $G'$ be the graph obtained from $G$ by adding every edge in $E(G[D]^c)$ to $G$. We say $D$ is \textit{inadmissible} in $G$ if every edge of $G'[D]$ is inadmissible in $G'$.

Given two graphs $G$ and $H$, their \textit{Cartesian product} $G \square H$ is the graph with vertex set $V(G) \times V(H)$ where two vertices $(v,w)$ and $(v',w')$ are adjacent if either $v=v'$ and $ww'\in E(H)$, or $w=w'$ and $vv'\in E(G)$. We write $n \equiv_k a$ for $n = a$ modulo $k$.

The following is a key lemma.

\begin{lemma}[Consequence of Corollary 5.13 in \cite{EPL82}]\label{lem:matchspacedim}

Let $G$ be a matchable graph on at least four vertices. Then
$$\pm(G) \ge \frac{|E(F(G))|-|V(G)|}{2}+2.$$

\end{lemma}

\section{A constant number of perfect matchings:\\ The 2- and 3-connected cases}\label{sec:2and3connPl}

\subsection{The 2-connected case}

\begin{theorem}
For every even $n\ge 6$ we have a planar $2$-connected $n$-vertex graph $G$ with exactly four perfect matchings, minimum degree~$3$, and size $\frac{3n}{2}+1$. The constant $4$ is best possible.
\end{theorem}

\begin{proof}
See \cref{fig:2conn4PMs}. It is elementary to verify that all of the theorem's statements hold, except for the last one. But this follows directly from Mader's result stating that every matchable 2-connected graph with minimum degree at least 3 except $K_4$ has at least four perfect matchings~\cite{M76}.

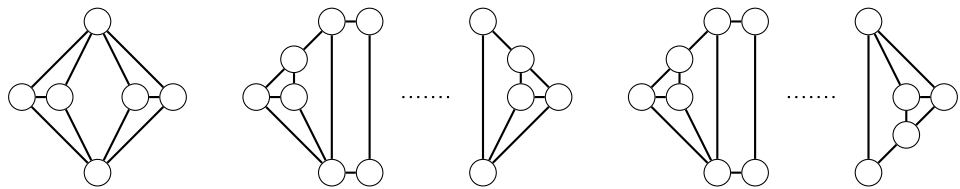
\begin{figure}[htb]
    \centering

\begin{tikzpicture}[scale=0.5,main_node/.style={circle,draw,minimum size=1em,inner sep=3pt]}]

\node[main_node] (llm) at (-2, 0) {};
\node[main_node] (lm) at (-1, 0) {};
\node[main_node] (mu) at (0, 2) {};
\node[main_node] (md) at (0, -2) {};
\node[main_node] (rm) at (1, 0) {};
\node[main_node] (rrm) at (2, 0) {};

 \path[draw, thick]
(llm) edge node {} (lm)
(llm) edge node {} (mu)
(llm) edge node {} (md)
(lm) edge node {} (mu)
(lm) edge node {} (md)
(rrm) edge node {} (rm)
(rrm) edge node {} (mu)
(rrm) edge node {} (md)
(rm) edge node {} (mu)
(rm) edge node {} (md)
;

\end{tikzpicture}
\hspace{0.5cm}
\begin{tikzpicture}[scale=0.5,main_node/.style={circle,draw,minimum size=1em,inner sep=3pt]}]

\node[main_node] (llm) at (-2, 0) {};
\node[main_node] (lm) at (-1, 0) {};
\node[main_node] (lu) at (-1, 1) {};
\node[main_node] (lau1) at (0, 2) {};
\node[main_node] (lad1) at (0, -2) {};
\node[main_node] (lau2) at (1, 2) {};
\node[main_node] (lad2) at (1, -2) {};

\node[main_node] (2lau1) at (4, 2) {};
\node[main_node] (2lad1) at (4, -2) {};

\node[main_node] (rm) at (5, 0) {};
\node[main_node] (rrm) at (6, 0) {};
\node[main_node] (ru) at (5, 1) {};

\node[main_node, draw=none] (rc) at (3.5, 0) {};
\node[main_node, draw=none] (lc) at (1.5, 0) {};

\path[draw, thick]
(llm) edge node {} (lm)
(llm) edge node {} (lad1)
(llm) edge node {} (lu)
(lm) edge node {} (lad1)
(lm) edge node {} (lu)
(lu) edge node {} (lau1)
(lad1) edge node {} (lau1)
(lad1) edge node {} (lad2)
(lau1) edge node {} (lau2)
(lad2) edge node {} (lau2)

(rrm) edge node {} (rm)
(rrm) edge node {} (ru)
(rrm) edge node {} (2lad1)
(rm) edge node {} (ru)
(rm) edge node {} (2lad1)
(ru) edge node {} (2lau1)
(2lad1) edge node {} (2lau1)

;

\path[draw, thick, dotted]
(lc) edge node {} (rc)
;

\end{tikzpicture}
\hspace{0.5cm}
\begin{tikzpicture}[scale=0.5,main_node/.style={circle,draw,minimum size=1em,inner sep=3pt]}]

\node[main_node] (llm) at (-2, 0) {};
\node[main_node] (lm) at (-1, 0) {};
\node[main_node] (lu) at (-1, 1) {};
\node[main_node] (lau1) at (0, 2) {};
\node[main_node] (lad1) at (0, -2) {};
\node[main_node] (lau2) at (1, 2) {};
\node[main_node] (lad2) at (1, -2) {};

\node[main_node] (2lau1) at (4, 2) {};
\node[main_node] (2lad1) at (4, -2) {};

\node[main_node] (rm) at (5, 0) {};
\node[main_node] (rrm) at (6, 0) {};
\node[main_node] (rd) at (5, -1) {};

\node[main_node, draw=none] (rc) at (3.5, 0) {};
\node[main_node, draw=none] (lc) at (1.5, 0) {};

\path[draw, thick]
(llm) edge node {} (lm)
(llm) edge node {} (lad1)
(llm) edge node {} (lu)
(lm) edge node {} (lad1)
(lm) edge node {} (lu)
(lu) edge node {} (lau1)
(lad1) edge node {} (lau1)
(lad1) edge node {} (lad2)
(lau1) edge node {} (lau2)
(lad2) edge node {} (lau2)

(rrm) edge node {} (rm)
(rrm) edge node {} (rd)
(rrm) edge node {} (2lau1)
(rm) edge node {} (rd)
(rm) edge node {} (2lau1)
(rd) edge node {} (2lad1)
(2lad1) edge node {} (2lau1)

;

\path[draw, thick, dotted]
(lc) edge node {} (rc)
;

\end{tikzpicture}

\caption{A planar 2-connected graph with four perfect matchings, minimum degree 3, size $\frac{3n}{2}+1$, and order $n=6$ (left), order $n\ge8, n \equiv_4 0$ (middle), and order $n\ge10, n \equiv_4 2$ (right).}
    \label{fig:2conn4PMs}
\end{figure}
\end{proof}

This first result is not hard to obtain. But the following points motivate its inclusion. (1)~It was surprising to us that there exist infinitely many planar 2-connected $n$-vertex graphs with minimum degree~3 and size $\frac{3n}{2}+1$ with a small positive constant number of perfect matchings. (2) Not only is our theorem best possible by Mader's result~\cite{M76}, he only provided one  graph proving the sharpness of his inequality as an example; we give infinitely many. (3) Our result also points to the fact that results of Mader have been misquoted in Theorem~1.6.6 on \cite[p.~33]{YL09}. 
Their ``theorem'' states that a matchable 2-connected non-bicritical\footnote{A graph $G$ is \textit{bicritical} if for every distinct $v,w \in V(G)$ the graph $G - v - w$ is matchable.} graph $G$ with $\delta(G)\ge k$ has at least $k!$ perfect matchings. This is not accurate for $k = 3$: our family of graphs, for $n \ge 8$, is not bicritical and neither does it satisfy $\pm(G) \ge \delta(G)!$. This omission does not seem to appear in the published Errata. We note that there are counterexamples for every even $k \ge 4$: Take the disjoint union of two copies of $K_{k+1}$. Take an edge $e_1=v_1w_1$ from the first copy and an edge $e_2=v_2w_2$ from the second copy. Remove $e_1, e_2$ and add the edges $v_1v_2$ and $w_1w_2$. The resulting graph has $2((k-1)!!)^2<k!$ perfect matchings.

\subsection{The 3-connected case}

On the one hand, it is well-known that planar 3-connected graphs of even order may have no perfect matching, see \cite{P92}. On the other hand, if they are 3-regular they have an exponential number of perfect matchings. Furthermore, existing results imply that matchable\footnote{Every even order planar 4-connected graph is matchable.} planar $4$-connected graphs have at least linearly many perfect matchings, see the next section. So initially, to us, the asymptotic behaviour of the number of perfect matchings in matchable planar 3-connected graphs was unclear.

One approach was to investigate whether planar 3-connected graphs are matching covered, as by Lemma~\ref{lem:matchspacedim} every matching covered $3$-connected graph has at least a linear number of perfect matchings. But it is not difficult to see that this need not be so, in a strong sense, as we now show. We first need a preparatory lemma.

\begin{lemma}\label{lem:inadmin}
Let $G$ be a matchable bipartite graph with bipartition $(A,B)$. For $X \in \{ E((G[A])^c),\allowbreak E((G[B])^c) \}$ and any $Y \subseteq X$, we have that $Y \cap E(F(G + Y))$ is empty.
\end{lemma}

\begin{proof}
    Assume there is an edge $e = vw \in Y$ in a perfect matching $M$ of $G + Y$. Without loss of generality, $v,w \in A$. We have $|A| = |B|$. For each $b \in B$ we have an edge $e_b \in M$ such that $b$ is incident with $e_b$; moreover, $|M| \ge |B|$. Since $e$ is also in $M$ but not incident with any $b \in B$, we have $|M| > \frac{|V(G)|}{2}$, a contradiction.
\end{proof}

\begin{proposition}\label{prop:nonCoveredSubgraph}
    For every outerplanar graph $H$, there exists a planar $3$-connected graph $G$ such that {\normalfont (i)} $\pm(G) \ge 2^{|V(G)|/4}$; {\normalfont (ii)} $H$ is an induced subgraph of $G$; and {\normalfont (iii)} no edge of $H$, now seen as a subgraph of $G$, is contained in a perfect matching of $G$.
\end{proposition}

\begin{proof}
Put $k := |V(H)|$. Consider $C_{2k} \Box K_2 =: G$, a cubic planar 3-connected graph of order~$n = 4k$. 
We have $\pm(G) \ge 2^{n/4}$. We embed this graph such that we have an  exterior $2k$-gon $P$, forming the boundary cycle of the infinite face. In $P$, a bipartite graph, we consider every second vertex, and add edges until $H$ is obtained. As $G$ is 3-connected, this graph must also be 3-connected; as $G$ is planar and $H$ is outerplanar, this graph must also be planar. By~\cref{lem:inadmin}, every edge of $H$ is inadmissible.
\end{proof}

\newcommand{\base}{H}
\newcommand{\baseDegFive}{I}
\newcommand{\baseTriang}{J}
\newcommand{\baseMinDegFour}{Q}
\newcommand{\gadgetMinDegFour}{R}

In the sequel we will use the following fact, sometimes implicitly.

\bigskip

\noindent \textbf{Fact.} \textit{Let $G$ be a matchable graph and let $C_1, \dots, C_k$ be the connected components of $F(G)$. Then 
$$\pm(G) = \prod_{i=1}^k \pm(C_i).$$}

\begin{lemma}\label{lem:glue}
Suppose $L$ is a matchable graph and $H$ is a matchable graph with an inadmissible set $D \subseteq V(H)$. Take the disjoint union of $L$ and $H$. Add an edge set $S$ between $D$ and $V(L)$.

Call the resulting graph $G$. Then $S$ is inadmissible and $\pm(G) = \pm(H) \cdot \pm(L).$
\end{lemma}

\begin{proof}
Since $L$ and $H$ are matchable, so is $G$. Let $M$ be a perfect matching in $G$. Let $T := M \cap S$. We have $|T|$ even as $H$ and $L$ have even order. Suppose $|T| > 0$. Let $X\subseteq D$ 
be the set of vertices in $H$ incident with an edge in $T$. Then $H - X$ has a perfect matching $N$. Let $H'$ be the graph obtained from $H$ by adding every missing edge between distinct vertices of $X$, so that the vertices in $X$ induce a clique in $H$. Then $N \cup Q$, with $Q$ a perfect matching of $H'[X]$, 
is a perfect matching of $H'$, contradicting the assumption that $D$ is inadmissible. The very last statement is due to the Fact stated just before the lemma.
\end{proof}

We now state and prove our main result. Its parts (i) and (ii) address the question how `close' to a cubic bridgeless graph a matchable planar 3-connected graph can be while having as few perfect matchings as possible, mirroring the 2-connected case we discussed earlier. For (iii) and (v) we treat the same problem  for maximally planar graphs. And in (iv) we discuss the bipartite case.

\begin{theorem}\label{prop:constantNumPMsClasses}
    Let $n_0$ and $c$ be integers. For every even $n\ge n_0$ there exists a planar $3$-connected $n$-vertex graph $G$ with $c$ perfect matchings, for the following combinations of $n_0$, $c$ and further restrictions on $G$:
    \begin{enumerate}[label=\textnormal{(\roman*)}]
        \item $n_0= 10$, $c=12$, $G$ has four vertices of degree $4$ and $n-4$ vertices of degree $3$; 
        \item $n_0= 14$, $c=18$, $G$ has two vertices of degree $5$ and $n-2$ vertices of degree $3$; 
        \item $n_0= 10$, $c=12$, and $G$ is a triangulation; 
        \item $n_0= 22$, $c=144$, $G$ is bipartite, has eight vertices of degree $4$ and $n-8$ vertices of degree $3$; 
        
        \item $n_0= 24$, $c=768$, $G$ has minimum degree $4$ and $G$ is a triangulation.
    \end{enumerate}

\end{theorem}

\begin{proof}
    We will give an infinite family of graphs for each case. The infinite families for (i), (iii), and (iv) make use of the planar bipartite graph $\base$ as drawn in~\cref{fig:BipPlanar12PMs}. By \cref{lem:inadmin}  any subset of edges from $\{ t_1t_2, t_1t_3, t_2t_3\}$ added to $\base$ is inadmissible. 

\begin{figure}[htb]
    \centering

\begin{tikzpicture}[scale=0.5,main_node/.style={circle,draw,minimum size=1em,inner sep=3pt]}]

\node[main_node, fill=blue!60] (1llm) at (-3, 0) {};
\node[main_node, fill=red] (1lm) at (-1, 0) {};
\node[main_node, fill=red] (1rm) at (1, 0) {};
\node[main_node, fill=red] (1rrm) at (3, 0) {$t_1$};
\node[main_node, fill=blue!60] (1mu) at (0, 1) {};
\node[main_node, fill=red] (1muu) at (0, 3) {$t_3$};
\node[main_node, fill=blue!60] (1md) at (0, -1) {};
\node[main_node, fill=red] (1mdd) at (0, -3) {$t_2$};
\node[main_node, fill=blue!60] (1ru) at (1.5, 1.5) {};
\node[main_node, fill=blue!60] (1rd) at (1.5, -1.5) {};

\node[main_node] (x) at (7, 0) {$x$};
\node[main_node] (z2) at (8.5, 1) {$z_2$};
\node[main_node] (z1) at (8.5, -1) {$z_1$};
\node[main_node] (z2n) at (10, 1) {};
\node[main_node] (z1n) at (10, -1) {};

\node[main_node] (z3) at (13.5, 1) {$z_3$};
\node[main_node] (z4) at (13.5, -1) {$z_4$};
\node[main_node] (y) at (15, 0) {$y$};

\node[main_node, draw=none] (lc) at (10.5,0) {};
\node[main_node, draw=none] (rc) at (13,0) {};

 \path[draw, thick]
(1llm) edge node {} (1muu) 
(1llm) edge node {} (1mdd) 
(1llm) edge node {} (1lm) 
(1lm) edge node {} (1mu)
(1lm) edge node {} (1md)
(1muu) edge node {} (1mu)
(1muu) edge node {} (1ru)
(1muu) edge node {} (1mu)
(1muu) edge node {} (1ru) 
(1rm) edge node {} (1ru) 
(1rm) edge node {} (1rd) 
(1rm) edge node {} (1md)
(1rm) edge node {} (1mu) 
(1mdd) edge node {} (1md) 
(1mdd) edge node {} (1rd) 
(1rrm) edge node {} (1rd) 
(1rrm) edge node {} (1ru)

(x) edge node {} (z1)
(z1) edge node {} (z2)
(z1) edge node {} (z1n)
(z2) edge node {} (z2n)
(z1n) edge node {} (z2n)
(z3) edge node {} (z4)
(z3) edge node {} (y)
(z4) edge node {} (y)

(x) edge[gray,dashed] node {} (1rrm)
(x) edge[gray,dashed] node {} (1mdd)
(z2) edge[gray,dashed] node {} (1rrm)
;

\path[draw, thick, dotted]
(lc) edge node {} (rc)
;

\draw[thick, rounded corners=10pt, gray, dashed]
  (1muu) .. controls ($(z3)+(0.5,2)$) ..  (y);

\end{tikzpicture}

\caption{On the left-hand side, with vertices coloured red and blue, a planar bipartite 10-vertex graph $\base$ with twelve perfect matchings. On the right-hand side, the graph $L_{n-10}^*$. The dashed edges connect $\base$ to $L_{n-10}^*$, making the entire depicted graph $G_{1,n}$ planar, 3-connected, and having twelve perfect matchings.
}
    \label{fig:BipPlanar12PMs}
\end{figure}
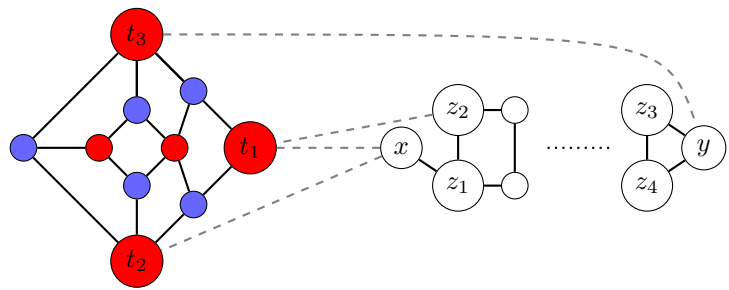

Consider, for an even positive integer $m$, the ladder graph $L_{m} := P_{m/2} \Box K_2$ of order $m\ge2$.

\bigskip

    \emph{Case 1. $n_0= 10$, $c=12$, $G$ has four vertices of degree $4$ and $n-4$ vertices of degree $3$.}

    First, suppose that $n=10$. Take $\base$ and add the two edges $t_1t_2$ and $t_1t_3$. The obtained graph $G_{1,10}$ still has 12 perfect matchings since $t_1t_2$ and $t_1t_3$ are inadmissible. Moreover, $G_{1,10}$ is 3-connected, planar, has 4 vertices of degree 4 and $n-4=6$ vertices of degree 3.\footnote{Note that we could delete either $t_1t_2$ or $t_1t_3$ and still obtain a planar 3-connected graph with 12 perfect matchings, which then only has 2 vertices of degree 4.}

    Now, suppose that $n=12$. Take the disjoint union of $\base$ with $K_2$ with $V(K_2) = \{x,y\}$, and add the edges $t_1x, t_1y, t_2x, t_3y$. The obtained graph $G_{1,12}$ is 3-connected, planar, has four vertices of degree $4$ and $n-4=8$ vertices of degree $3$. Furthermore, $\pm(G_{1,12})=12$ since $\pm(K_2)=1$ and $xy$ is the only added admissible edge. All other added edges are inadmissible due to the inadmissibility of $\{t_1,t_2,t_3\}$, by \cref{lem:glue}.

    Lastly, suppose that $n\ge14$ and $n$ is even. In $L_m$, let $z_1,z_2$ be the two adjacent vertices of degree 2 on one end of the ladder and $z_3,z_4$ the ones on the other end. For $m=2$, the pairs coincide, so $z_1=z_3$ and $z_2=z_4$. We define the graph $L_m^*$ for $m\ge4$ as the ladder graph $L_{m-2}$ where we add two vertices $x$, $y$ and add three edges $xz_1$, $yz_3$, $yz_4$. Note that $\pm(L_m^*)=1$.  Consider now $\base$ to be embedded as shown in \cref{fig:BipPlanar12PMs}. Take $G_{1,n}$ as $\base$ where we place $L_{n-10}^*$ in the face whose boundary cycle contains $t_1,t_2$ and $t_3$, and connect it to $\base$ with the edges $t_1z_2$, $t_1x$, $t_2x$, $t_3y$. The graph $G_{1,n}$, depicted in \cref{fig:BipPlanar12PMs}, is a 3-connected (this can be verified by a routine argument using Menger's Theorem) planar $n$-vertex graph that has four vertices of degree $4$, $n-4$ vertices of degree $3$, and has $\pm(\base) \cdot \pm(L_{n-10}^*)=12$ perfect matchings by \cref{lem:glue}.

\bigskip

    \emph{Case 2. $n_0= 14$, $c=18$, $G$ has two vertices of degree $5$ and $n-2$ vertices of degree $3$.}

    The proof is analogous to the one of Case 1, except that we replace $\base$ by $\baseDegFive$ from \cref{fig:BipPlanar18PMs}, which leads to a graph with 18 perfect matchings (instead of 12), two vertices of degree 5 and $n-2$ vertices of degree $3$.\\

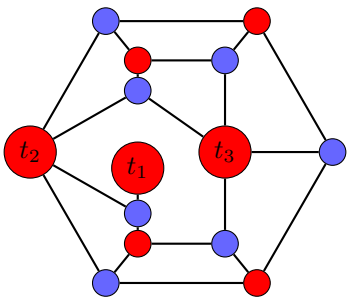
\begin{figure}[htb]
    \centering
\begin{tikzpicture}[scale=0.5,main_node/.style={circle,draw,minimum size=1em,inner sep=3pt]}]

\def\radius{4.0}

\node[main_node, fill=blue!60] (0) at (1*\radius, 0) {};
\node[main_node, fill=red] (1) at (0.5*\radius, {0.5*sqrt(3)*\radius}) {};
\node[main_node, fill=blue!60] (2) at (-0.5*\radius, {sqrt(3)*0.5*\radius}) {};
\node[main_node, fill=red] (3) at (-1*\radius, 0) {$t_2$};
\node[main_node, fill=blue!60] (4) at (-0.5*\radius, {sqrt(3)*-0.5*\radius}) {};
\node[main_node, fill=red] (5) at (0.5*\radius, {sqrt(3)*-0.5*\radius}) {};

\node[main_node, fill=red] (6) at ({\radius/sqrt(3)/2}, 0) {$t_3$};
\node[main_node, fill=blue!60] (7) at ({-\radius/sqrt(3)/2}, {sqrt(3)*0.35*\radius-0.2*\radius}) {};

\node[main_node, fill=blue!60] (8) at ({\radius/sqrt(3)/2}, {sqrt(3)*0.35*\radius}) {};
\node[main_node, fill=red] (9) at ({-\radius/sqrt(3)/2}, {sqrt(3)*0.35*\radius}) {};

\node[main_node, fill=blue!60] (10) at ({\radius/sqrt(3)/2}, {-sqrt(3)*0.35*\radius}) {};
\node[main_node, fill=red] (11) at ({-\radius/sqrt(3)/2}, {-sqrt(3)*0.35*\radius}) {};

\node[main_node, fill=red] (12) at ({-\radius/sqrt(3)/2}, {-sqrt(3)*0.35*\radius+0.5*\radius)}) {$t_1$};
\node[main_node, fill=blue!60] (13) at ({-\radius/sqrt(3)/2}, {-sqrt(3)*0.35*\radius+0.2*\radius}) {};

 \path[draw, thick]
(0) edge node {} (1) 
(1) edge node {} (2) 
(2) edge node {} (3) 
(3) edge node {} (4) 
(4) edge node {} (5) 
(5) edge node {} (0)

(0) edge node {} (6) 
(3) edge node {} (7) 
(6) edge node {} (7)

(8) edge node {} (9)
(8) edge node {} (1)
(9) edge node {} (2)
(8) edge node {} (6)
(9) edge node {} (7)

(10) edge node {} (11)
(10) edge node {} (5)
(11) edge node {} (4)
(10) edge node {} (6)

(12) edge node {} (13)
(11) edge node {} (13)
(13) edge node {} (3)

;

\end{tikzpicture}

\caption{A planar bipartite graph $\baseDegFive$ with 18 perfect matchings. A 2-colouring is shown.}
    \label{fig:BipPlanar18PMs}
\end{figure}

    \emph{Case 3. $n_0= 10$, $c=12$ and $G$ is a triangulation.}

    Triangulate the graph $\base$ by adding edges between the vertices of the red partite set. This triangulation $\baseTriang$ is a planar 3-connected 10-vertex graph. All the added edges are inadmissible by \cref{lem:inadmin}, hence, $\pm(\baseTriang)=\pm(\base)=12$, which solves this case for $n=10$.

    Suppose now that $n\ge12$ and $n$ is even. Consider the triangulation $\baseTriang$ with the inadmissible set $\{t_1, t_2, t_3\}$. Add a $K_2$ with $V(K_2) = \{x,y\}$ to $\baseTriang$ and the edges $t_1x, t_1y, t_2x, t_2y, t_3y$ (which is the same for Case 1 where $n=12$, but we added the edge $t_2y$ as well). The obtained graph $G_{3,12}$ is a 3-connected planar triangulation and has 12 perfect matchings since $xy$ is the only added admissible edge. Now, $\{t_1, t_2, x\}$ is an inadmissible set and we can recursively apply the same operation of adding $K_2$ while maintaining the same number of perfect matchings. So, Case 3 holds for any even $n\ge 12$.\\

    \emph{Case 4. $n_0= 22$, $c=144$, $G$ is bipartite, has eight vertices of degree $4$ and $n-8$ vertices of degree~$3$.}
    
    Consider two copies $\base^{(1)}, \base^{(2)}$ of the graph $\base$, where a vertex $v$ from $\base$ is denoted as $v^{(1)}$ in $\base^{(1)}$ and $v^{(2)}$ in $\base^{(2)}$.

    Suppose $n=22$.\footnote{We can also make a bipartite $20$-vertex graph with six vertices of degree 4, $n-6$ of degree $3$ and 144 perfect matchings. This is done by adding $t_i^{(1)}t_i^{(2)}$ for every $i \in \{1,2,3\}$. See \cref{fig:bip_n20}.} Consider the graph $G_{4,22}$ consisting of $\base^{(1)}, \base^{(2)}$, a $K_2$ with $V(K_2) = \{x^{(1)},x^{(2)}\}$, and the edges $t_1^{(1)}t_1^{(2)}, t_1^{(1)}x^{(1)}, t_1^{(2)}x^{(2)}, t_3^{(1)}x^{(1)}, t_3^{(2)}x^{(2)}, t_2^{(1)}t_2^{(2)}$. The graph $G_{4,22}$, depicted in \cref{fig:bip_n22}, is 3-connected, planar, bipartite, has eight vertices of degree 4 and $n-8=14$ vertices of degree 3.

\def\graphsep{18cm}

\begin{figure}[htb]
    \centering
\begin{tikzpicture}[scale=0.5,main_node/.style={circle,draw,minimum size=1em,inner sep=3pt]}]

\node[main_node, fill=blue!60] (1llm) at (-3, 0) {};
\node[main_node, fill=red] (1lm) at (-1, 0) {};
\node[main_node, fill=red] (1rm) at (1, 0) {};
\node[main_node, fill=red] (1rrm) at (3, 0) {$t_1^{(1)}$};
\node[main_node, fill=blue!60] (1mu) at (0, 1) {};
\node[main_node, fill=red] (1muu) at (0, 3) {$t_3^{(1)}$};
\node[main_node, fill=blue!60] (1md) at (0, -1) {};
\node[main_node, fill=red] (1mdd) at (0, -3) {$t_2^{(1)}$};
\node[main_node, fill=blue!60] (1ru) at (1.5, 1.5) {};
\node[main_node, fill=blue!60] (1rd) at (1.5, -1.5) {};

 \path[draw, thick]
(1llm) edge node {} (1muu) 
(1llm) edge node {} (1mdd) 
(1llm) edge node {} (1lm) 
(1lm) edge node {} (1mu)
(1lm) edge node {} (1md)
(1muu) edge node {} (1mu)
(1muu) edge node {} (1ru)
(1muu) edge node {} (1mu)
(1muu) edge node {} (1ru) 
(1rm) edge node {} (1ru) 
(1rm) edge node {} (1rd) 
(1rm) edge node {} (1md)
(1rm) edge node {} (1mu) 
(1mdd) edge node {} (1md) 
(1mdd) edge node {} (1rd) 
(1rrm) edge node {} (1rd) 
(1rrm) edge node {} (1ru) 
;

\begin{scope}[xshift=\graphsep, xscale=-1]
  
  \node[main_node, fill=red] (2llm)  at (-3,  0) {};
  \node[main_node, fill=blue!60]  (2lm)   at (-1,  0) {};
  \node[main_node, fill=blue!60]  (2rm)   at ( 1,  0) {};
  \node[main_node, fill=blue!60]  (2rrm)  at ( 3,  0) {$t_1^{(2)}$};
  \node[main_node, fill=red] (2mu)   at ( 0,  1) {};
  \node[main_node, fill=blue!60]  (2muu)  at ( 0,  3) {$t_3^{(2)}$};
  \node[main_node, fill=red] (2md)   at ( 0, -1) {};
  \node[main_node, fill=blue!60]  (2mdd)  at ( 0, -3) {$t_2^{(2)}$};
  \node[main_node, fill=red] (2ru)   at ( 1.5,  1.5) {};
  \node[main_node, fill=red] (2rd)   at ( 1.5, -1.5) {};

  \path[draw, thick]
   (2llm) edge (2muu)
   (2llm) edge (2mdd)
   (2llm) edge (2lm)
   (2lm)  edge (2mu)
   (2lm)  edge (2md)
   (2muu) edge (2mu)
   (2muu) edge (2ru)
   (2muu) edge (2mu)
   (2muu) edge (2ru)
   (2rm)  edge (2ru)
   (2rm)  edge (2rd)
   (2rm)  edge (2md)
   (2rm)  edge (2mu)
   (2mdd) edge (2md)
   (2mdd) edge (2rd)
   (2rrm) edge (2rd)
   (2rrm) edge (2ru)
  ;
\end{scope}

\path[draw, thick]
(1rrm) edge node {} (2rrm) 
(1muu) edge node {} (2muu) 
(1mdd) edge node {} (2mdd) 
;

\end{tikzpicture}
\caption{A planar bipartite 3-connected $20$-vertex graph with 144 perfect matchings.}
    \label{fig:bip_n20}
\end{figure}
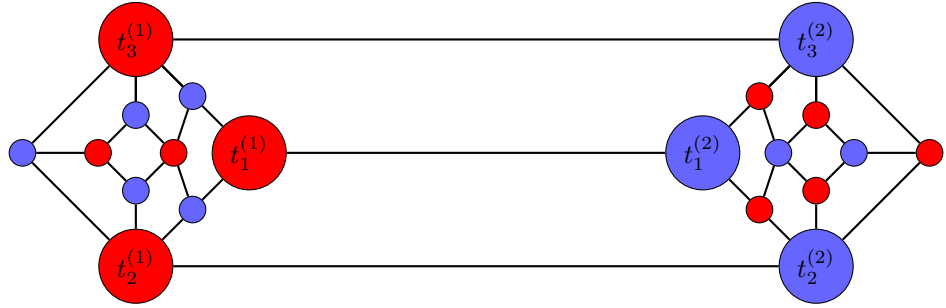

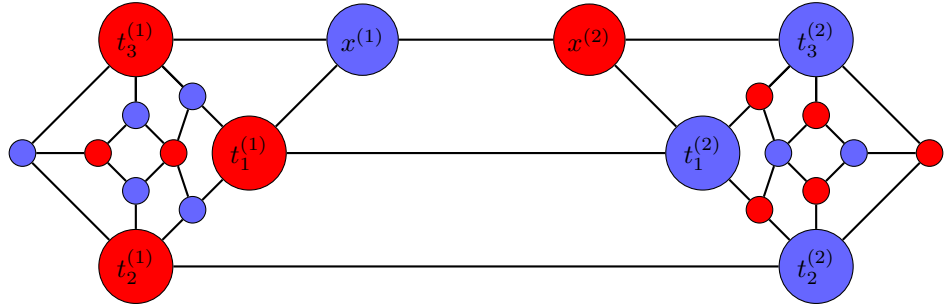
\begin{figure}[htb]
    \centering
\begin{tikzpicture}[scale=0.5,main_node/.style={circle,draw,minimum size=1em,inner sep=3pt]}]

\node[main_node, fill=blue!60] (1llm) at (-3, 0) {};
\node[main_node, fill=red] (1lm) at (-1, 0) {};
\node[main_node, fill=red] (1rm) at (1, 0) {};
\node[main_node, fill=red] (1rrm) at (3, 0) {$t_1^{(1)}$};
\node[main_node, fill=blue!60] (1mu) at (0, 1) {};
\node[main_node, fill=red] (1muu) at (0, 3) {$t_3^{(1)}$};
\node[main_node, fill=blue!60] (1md) at (0, -1) {};
\node[main_node, fill=red] (1mdd) at (0, -3) {$t_2^{(1)}$};
\node[main_node, fill=blue!60] (1ru) at (1.5, 1.5) {};
\node[main_node, fill=blue!60] (1rd) at (1.5, -1.5) {};

 \path[draw, thick]
(1llm) edge node {} (1muu) 
(1llm) edge node {} (1mdd) 
(1llm) edge node {} (1lm) 
(1lm) edge node {} (1mu)
(1lm) edge node {} (1md)
(1muu) edge node {} (1mu)
(1muu) edge node {} (1ru)
(1muu) edge node {} (1mu)
(1muu) edge node {} (1ru) 
(1rm) edge node {} (1ru) 
(1rm) edge node {} (1rd) 
(1rm) edge node {} (1md)
(1rm) edge node {} (1mu) 
(1mdd) edge node {} (1md) 
(1mdd) edge node {} (1rd) 
(1rrm) edge node {} (1rd) 
(1rrm) edge node {} (1ru) 
;

\node[main_node, fill=blue!60] (x) at (\graphsep/3, 3) {$x^{(1)}$};
\node[main_node, fill=red] (y) at (2*\graphsep/3, 3) {$x^{(2)}$};

\begin{scope}[xshift=\graphsep, xscale=-1]
  
  \node[main_node, fill=red] (2llm)  at (-3,  0) {};
  \node[main_node, fill=blue!60]  (2lm)   at (-1,  0) {};
  \node[main_node, fill=blue!60]  (2rm)   at ( 1,  0) {};
  \node[main_node, fill=blue!60]  (2rrm)  at ( 3,  0) {$t_1^{(2)}$};
  \node[main_node, fill=red] (2mu)   at ( 0,  1) {};
  \node[main_node, fill=blue!60]  (2muu)  at ( 0,  3) {$t_3^{(2)}$};
  \node[main_node, fill=red] (2md)   at ( 0, -1) {};
  \node[main_node, fill=blue!60]  (2mdd)  at ( 0, -3) {$t_2^{(2)}$};
  \node[main_node, fill=red] (2ru)   at ( 1.5,  1.5) {};
  \node[main_node, fill=red] (2rd)   at ( 1.5, -1.5) {};

  \path[draw, thick]
   (2llm) edge (2muu)
   (2llm) edge (2mdd)
   (2llm) edge (2lm)
   (2lm)  edge (2mu)
   (2lm)  edge (2md)
   (2muu) edge (2mu)
   (2muu) edge (2ru)
   (2muu) edge (2mu)
   (2muu) edge (2ru)
   (2rm)  edge (2ru)
   (2rm)  edge (2rd)
   (2rm)  edge (2md)
   (2rm)  edge (2mu)
   (2mdd) edge (2md)
   (2mdd) edge (2rd)
   (2rrm) edge (2rd)
   (2rrm) edge (2ru)
  ;
\end{scope}

\path[draw, thick]
(1rrm) edge node {} (2rrm) 
(1rrm) edge node {} (x) 
(2rrm) edge node {} (y) 
(1muu) edge node {} (x) 
(2muu) edge node {} (y) 
(x) edge node {} (y)
(1mdd) edge node {} (2mdd) 
;

\end{tikzpicture}
\caption{A planar bipartite 3-connected $22$-vertex graph with 144 perfect matchings.}
    \label{fig:bip_n22}
\end{figure}

    By \cref{lem:inadmin} the set  $\{t_1^{(1)}, t_2^{(1)}, t_3^{(1)}\}$ is inadmissible in $H^{(1)}$ and $\{t_1^{(2)}, t_2^{(2)}, t_3^{(2)}\}$ is inadmissible in $H^{(2)}$. 
    Apply \cref{lem:glue} with $H = H^{(1)}$ and $L = K_2$. When doing so, we add the edges $t_3^{(1)} x^{(1)}$ and $t_1^{(1)} x^{(1)}$. Let us call the resulting graph $R$. We then apply \cref{lem:glue} with $H = H^{(2)}$ and $L = R$. Thus, all added edges except $x^{(1)}x^{(2)}$ are inadmissible. Since $\pm(\base) = 12$ and $\pm(K_2) = 1$, we have $\pm(G_{4,22})= 144$.

    Suppose $n\ge24$ and $n \equiv_4 0$, which stands for $n = 0$ modulo 4. Consider the ladder graph $L_{n-18}$ with $z^{(1)}$ and $z^{(2)}$ as two vertices of degree 2 in $L_{n-18}$ which are at distance $\frac{n-18}{2}$ (this is the diameter of $L_{n-18}$) from each other. Now, construct the graph from $\base^{(1)}, \base^{(2)}$, and $L_{n-18}$, where we identify $z^{(1)}$ with $t_1^{(1)}$ and $z^{(2)}$ with $t_1^{(2)}$. 
    Additionally, add the three edges $t_2^{(1)}y^{(1)},t_3^{(2)}y^{(2)}, t_3^{(1)}t_2^{(2)}$, where $y^{(1)}$ and $y^{(2)}$ are the degree 2 neighbours of $z^{(1)}$ and $z^{(2)}$ in $L_{n-18}$, respectively. The resulting graph $G_{4,n \equiv_4 0}$, depicted in \cref{fig:bip_n0mod4}, is 3-connected, planar, bipartite, has order $n$ with $n\ge24$ and $n \equiv_4 0$, has 8 vertices of degree 4 and $n-8$ vertices of degree 3.

    \bigskip

\begin{figure}[htb]
    \centering
    
\begin{tikzpicture}[scale=0.5,main_node/.style={circle,draw,minimum size=1em,inner sep=3pt]}]

\path[use as bounding box] (-6cm,-4.5cm) rectangle (\graphsep+6cm,4.5cm);

\node[main_node, fill=blue!60] (1llm) at (-3, 0) {};
\node[main_node, fill=red] (1lm) at (-1, 0) {};
\node[main_node, fill=red] (1rm) at (1, 0) {};
\node[main_node, fill=red] (1rrm) at (3, 0) {$t_1^{(1)}$};
\node[main_node, fill=blue!60] (1mu) at (0, 1) {};
\node[main_node, fill=red] (1muu) at (0, 3) {$t_3^{(1)}$};
\node[main_node, fill=blue!60] (1md) at (0, -1) {};
\node[main_node, fill=red] (1mdd) at (0, -3) {$t_2^{(1)}$};
\node[main_node, fill=blue!60] (1ru) at (1.5, 1.5) {};
\node[main_node, fill=blue!60] (1rd) at (1.5, -1.5) {};

 \path[draw, thick]
(1llm) edge node {} (1muu) 
(1llm) edge node {} (1mdd) 
(1llm) edge node {} (1lm) 
(1lm) edge node {} (1mu)
(1lm) edge node {} (1md)
(1muu) edge node {} (1mu)
(1muu) edge node {} (1ru)
(1muu) edge node {} (1mu)
(1muu) edge node {} (1ru) 
(1rm) edge node {} (1ru) 
(1rm) edge node {} (1rd) 
(1rm) edge node {} (1md)
(1rm) edge node {} (1mu) 
(1mdd) edge node {} (1md) 
(1mdd) edge node {} (1rd) 
(1rrm) edge node {} (1rd) 
(1rrm) edge node {} (1ru) 
;

\node[main_node, fill=blue!60] (u1) at (\graphsep/4, 1) {};
\node[main_node, fill=blue!60] (d1) at (\graphsep/4, -1) {};
\node[main_node, fill=red] (u2) at (\graphsep/4+1.5cm, 1) {};
\node[main_node, fill=red] (d2) at (\graphsep/4+1.5cm, -1) {};

\node[main_node, fill=blue!60] (u3) at (\graphsep/4+3.0cm, 1) {};
\node[main_node, fill=blue!60] (d3) at (\graphsep/4+3.0cm, -1) {};
\node[main_node, fill=red] (u4) at (\graphsep/4+4.5cm, 1) {};
\node[main_node, fill=red] (d4) at (\graphsep/4+4.5cm, -1) {};

\node[main_node, draw=none] (m4) at (\graphsep/4+5.0cm, 0) {};

\node[main_node, fill=red] (u-1) at (3*\graphsep/4, 1) {};
\node[main_node, fill=red] (d-1) at (3*\graphsep/4, -1) {};
\node[main_node, fill=blue!60] (u-2) at (3*\graphsep/4-1.5cm, 1) {};
\node[main_node, fill=blue!60] (d-2) at (3*\graphsep/4-1.5cm, -1) {};

\node[main_node, draw=none] (m-2) at (3*\graphsep/4-2.0cm, 0) {};

\begin{scope}[xshift=\graphsep, xscale=-1]
  
  \node[main_node, fill=red] (2llm)  at (-3,  0) {};
  \node[main_node, fill=blue!60]  (2lm)   at (-1,  0) {};
  \node[main_node, fill=blue!60]  (2rm)   at ( 1,  0) {};
  \node[main_node, fill=blue!60]  (2rrm)  at ( 3,  0) {$t_1^{(2)}$};
  \node[main_node, fill=red] (2mu)   at ( 0,  1) {};
  \node[main_node, fill=blue!60]  (2muu)  at ( 0,  3) {$t_3^{(2)}$};
  \node[main_node, fill=red] (2md)   at ( 0, -1) {};
  \node[main_node, fill=blue!60]  (2mdd)  at ( 0, -3) {$t_2^{(2)}$};
  \node[main_node, fill=red] (2ru)   at ( 1.5,  1.5) {};
  \node[main_node, fill=red] (2rd)   at ( 1.5, -1.5) {};

  \path[draw, thick]
   (2llm) edge (2muu)
   (2llm) edge (2mdd)
   (2llm) edge (2lm)
   (2lm)  edge (2mu)
   (2lm)  edge (2md)
   (2muu) edge (2mu)
   (2muu) edge (2ru)
   (2muu) edge (2mu)
   (2muu) edge (2ru)
   (2rm)  edge (2ru)
   (2rm)  edge (2rd)
   (2rm)  edge (2md)
   (2rm)  edge (2mu)
   (2mdd) edge (2md)
   (2mdd) edge (2rd)
   (2rrm) edge (2rd)
   (2rrm) edge (2ru)
  ;
\end{scope}

\path[draw, thick]
(1rrm) edge node {} (u1) 
(1rrm) edge node {} (d1) 
(1mdd) edge node {} (d1)

(u1) edge node {} (u2) 
(u1) edge node {} (d2) 
(d1) edge node {} (d2)

(u2) edge node {} (u3) 
(u2) edge node {} (d3) 
(d2) edge node {} (d3)

(u3) edge node {} (u4) 
(u3) edge node {} (d4) 
(d3) edge node {} (d4)

(2rrm) edge node {} (u-1) 
(2rrm) edge node {} (d-1) 
(2muu) edge node {} (u-1)

(u-1) edge node {} (u-2) 
(d-1) edge node {} (u-2) 
(d-1) edge node {} (d-2) 

;

\path[draw, thick, dotted]
(m4) edge node {} (m-2)

;

\draw[thick, rounded corners=10pt]
  (1muu) .. controls ($(2muu)+(5,7)$) and ($(2llm)+(7,-2)$) ..  (2mdd);

\end{tikzpicture}

\caption{A planar bipartite 3-connected $n$-vertex graph with $n\equiv_40$ and 144 perfect matchings.}
    \label{fig:bip_n0mod4}
\end{figure}
    
    We now prove that $\pm(G_{4,n \equiv_4 0}) = 144.$ Let $L' := G_{4,n \equiv_4 0} - H^{(1)} - H^{(2)}$. First apply \cref{lem:glue} with $H = H^{(1)}$ and $L = L'$; here $D = \{ t_1^{(1)}, t_2^{(1)} \}$. We denote the resulting graph by $H'$. We then apply \cref{lem:glue} with $H = H^{(2)}$ and $L = H'$; here $D = \{ t_1^{(2)}, t_2^{(2)}, t_3^{(2)} \}$. Since the induced subgraph $G_{4,n \equiv_4 0}[V(L_{n-18})\setminus \{t_1^{(1)}, t_1^{(2)}\}]$, which is isomorphic to $L'$, only has one perfect matching, it follows that $\pm(G_{4,n \equiv_4 0})=\pm(\base)^2=144$.

    Suppose $n\ge24$ and $n \equiv_4 2$. Similar to the case $n \equiv_4 0$, consider $L_{n-20}$ with $z^{(1)}$ and $z^{(2)}$ as two vertices of degree 2 in $L_{n-20}$ which are at distance $\frac{n-20}{2}$. Consider the graph constructed from $\base^{(1)}, \base^{(2)}$, and $L_{n-20}$, where we identify $z^{(1)}$ with $t_1^{(1)}$ and $z^{(2)}$ with $t_1^{(2)}$. Take $y^{(1)}$ and $y^{(2)}$ as the degree 2 neighbours of $z^{(1)}$ and $z^{(2)}$ in $L_{n-20}$, respectively. Take $w^{(2)}$ as the degree 3 neighbour of $z^{(2)}$ in $L_{n-20}$.  Delete the edge $w^{(2)} t_1^{(2)}$.
    In contrast with the previous construction, we now add a $K_2$ with $V(K_2) = \{x^{(1)},x^{(2)}\}$, and also the edges $t_2^{(1)}y^{(1)}, t_3^{(1)} x^{(1)}, t_1^{(2)} x^{(2)}, t_2^{(2)} w^{(2)},  t_3^{(2)} x^{(2)}, y^{(2)} x^{(1)} $. 
    The resulting graph $G_{4,n \equiv_4 2}$, depicted in \cref{fig:bip_n2mod4}, is 3-connected, planar, bipartite, has order $n$ with $n\ge24$ and $n \equiv_4 2$, has 8 vertices of degree 4 and $n-8$ vertices of degree 3.

    \begin{figure}[htb]
    \centering
    
\begin{tikzpicture}[scale=0.5,main_node/.style={circle,draw,minimum size=1em,inner sep=3pt]}]

\node[main_node, fill=blue!60] (1llm) at (-3, 0) {};
\node[main_node, fill=red] (1lm) at (-1, 0) {};
\node[main_node, fill=red] (1rm) at (1, 0) {};
\node[main_node, fill=red] (1rrm) at (3, 0) {$t_1^{(1)}$};
\node[main_node, fill=blue!60] (1mu) at (0, 1) {};
\node[main_node, fill=red] (1muu) at (0, 3) {$t_3^{(1)}$};
\node[main_node, fill=blue!60] (1md) at (0, -1) {};
\node[main_node, fill=red] (1mdd) at (0, -3) {$t_2^{(1)}$};
\node[main_node, fill=blue!60] (1ru) at (1.5, 1.5) {};
\node[main_node, fill=blue!60] (1rd) at (1.5, -1.5) {};

 \path[draw, thick]
(1llm) edge node {} (1muu) 
(1llm) edge node {} (1mdd) 
(1llm) edge node {} (1lm) 
(1lm) edge node {} (1mu)
(1lm) edge node {} (1md)
(1muu) edge node {} (1mu)
(1muu) edge node {} (1ru)
(1muu) edge node {} (1mu)
(1muu) edge node {} (1ru) 
(1rm) edge node {} (1ru) 
(1rm) edge node {} (1rd) 
(1rm) edge node {} (1md)
(1rm) edge node {} (1mu) 
(1mdd) edge node {} (1md) 
(1mdd) edge node {} (1rd) 
(1rrm) edge node {} (1rd) 
(1rrm) edge node {} (1ru) 
;

\node[main_node, fill=blue!60] (u1) at (\graphsep/4, 1) {};
\node[main_node, fill=blue!60] (d1) at (\graphsep/4, -1) {};
\node[main_node, fill=red] (u2) at (\graphsep/4+1.5cm, 1) {};
\node[main_node, fill=red] (d2) at (\graphsep/4+1.5cm, -1) {};

\node[main_node, fill=blue!60] (u3) at (\graphsep/4+3.0cm, 1) {};
\node[main_node, fill=blue!60] (d3) at (\graphsep/4+3.0cm, -1) {};
\node[main_node, fill=red] (u4) at (\graphsep/4+4.5cm, 1) {};
\node[main_node, fill=red] (d4) at (\graphsep/4+4.5cm, -1) {};

\node[main_node, draw=none] (m4) at (\graphsep/4+5cm, 0) {};

\node[main_node, fill=red] (u-1) at (3*\graphsep/4, 1) {};
\node[main_node, fill=red] (d-1) at (3*\graphsep/4, -1) {};
\node[main_node, fill=blue!60] (u-2) at (3*\graphsep/4-1.5cm, 1) {};
\node[main_node, fill=blue!60] (d-2) at (3*\graphsep/4-1.5cm, -1) {};

\node[main_node, fill=red] (uu-1) at (\graphsep-3cm, 3) {$x^{(2)}$};
\node[main_node, fill=blue!60] (uu-2) at (3*\graphsep/4-1.5cm, 3) {$x^{(1)}$};

\node[main_node, draw=none] (m-2) at (3*\graphsep/4-2cm, 0) {};

\begin{scope}[xshift=\graphsep, xscale=-1]
  
  \node[main_node, fill=red] (2llm)  at (-3,  0) {};
  \node[main_node, fill=blue!60]  (2lm)   at (-1,  0) {};
  \node[main_node, fill=blue!60]  (2rm)   at ( 1,  0) {};
  \node[main_node, fill=blue!60]  (2rrm)  at ( 3,  0) {$t_1^{(2)}$};
  \node[main_node, fill=red] (2mu)   at ( 0,  1) {};
  \node[main_node, fill=blue!60]  (2muu)  at ( 0,  3) {$t_3^{(2)}$};
  \node[main_node, fill=red] (2md)   at ( 0, -1) {};
  \node[main_node, fill=blue!60]  (2mdd)  at ( 0, -3) {$t_2^{(2)}$};
  \node[main_node, fill=red] (2ru)   at ( 1.5,  1.5) {};
  \node[main_node, fill=red] (2rd)   at ( 1.5, -1.5) {};

  \path[draw, thick]
   (2llm) edge (2muu)
   (2llm) edge (2mdd)
   (2llm) edge (2lm)
   (2lm)  edge (2mu)
   (2lm)  edge (2md)
   (2muu) edge (2mu)
   (2muu) edge (2ru)
   (2muu) edge (2mu)
   (2muu) edge (2ru)
   (2rm)  edge (2ru)
   (2rm)  edge (2rd)
   (2rm)  edge (2md)
   (2rm)  edge (2mu)
   (2mdd) edge (2md)
   (2mdd) edge (2rd)
   (2rrm) edge (2rd)
   (2rrm) edge (2ru)
  ;
\end{scope}

\path[draw, thick]
(1rrm) edge node {} (u1) 
(1rrm) edge node {} (d1) 
(1mdd) edge node {} (d1)

(u1) edge node {} (u2) 
(u1) edge node {} (d2) 
(d1) edge node {} (d2)

(u2) edge node {} (u3) 
(u2) edge node {} (d3) 
(d2) edge node {} (d3)

(u3) edge node {} (u4) 
(u3) edge node {} (d4) 
(d3) edge node {} (d4)

(2rrm) edge node {} (u-1) 
(2mdd) edge node {} (d-1)

(u-1) edge node {} (u-2) 
(d-1) edge node {} (u-2) 
(d-1) edge node {} (d-2) 

(uu-1) edge node {} (2muu)
(uu-1) edge node {} (2rrm)
(uu-1) edge node {} (uu-2)
(uu-2) edge node {} (1muu)
(uu-2) edge node {} (u-1)

;

\path[draw, thick, dotted]
(m4) edge node {} (m-2)

;

\end{tikzpicture}

\caption{A planar bipartite 3-connected $n$-vertex graph with $n\equiv_42$ and 144 perfect matchings.}
    \label{fig:bip_n2mod4}
\end{figure}

    We now prove that $\pm(G_{4,n \equiv_4 2}) = 144.$ Let $L'' := G_{4,n \equiv_4 2} - H^{(1)} - H^{(2)}$. First apply \cref{lem:glue} with $H = H^{(1)}$ and $L = L''$; here $D = \{ t_1^{(1)}, t_2^{(1)}, t_3^{(1)} \}$. We denote the resulting graph by $H''$. We then apply \cref{lem:glue} with $H = H^{(2)}$ and $L = H''$; here $D = \{ t_1^{(2)}, t_2^{(2)}, t_3^{(2)} \}$. Since $L''$ only has one perfect matching, it follows that $\pm(G_{4,n \equiv_4 2})=144$.

    \bigskip

    \emph{Case 5. $n_0= 24$, $c=768$ and $G$ has minimum degree $4$ and $G$ is a triangulation.}

     Consider $R_4$ as a $C_4$ with consecutive vertices $z_1, z_2, z_3, z_4$ to which we add the edge $z_2 z_4$. Consider $R_m$ with even $m\ge 6$ as the disjoint union of $R_4$ and a path graph $P_{m-4}$ with consecutive vertices $p_1, \dots, p_{m-4}$ where we add the edge $z_1 p_{m-4}$. Note that $\pm(R_m)=2$. 
     
     Suppose $n\ge 24$ is even and take $G_{5,n}$ as $\baseMinDegFour$ as drawn in \cref{fig:minDeg4}. The graph $\baseMinDegFour$ can be obtained from the graph $\base$ in two steps. First, just like in Case 3, we triangulate $\base$ by adding edges between the vertices of the red partite set. We call the resulting graph $\baseTriang$. Second, we replace each vertex of the blue partite set by a triangle and triangulate the graph by adding two edges to each vertex of an added triangle. The obtained graph is $\baseMinDegFour$. By doing the replacement operation for exactly one blue vertex, the resulting graph has twice the number of perfect matchings of the original graph. Since there are five blue vertices in $\baseTriang$ and because $\pm(\baseTriang)=12$, we thus have $\pm(\baseMinDegFour)=12 \cdot 2^5=384$. Furthermore, each inadmissible set of three vertices in $\baseTriang$ stays inadmissible in $\baseMinDegFour$, in particular $\{t_1,t_2,t_3\}$ stays inadmissible.  Now, place $R_{n-20}$ in the common face of vertices $t_1,t_2,t_3$. Add the edge $t_1 z_1$ in case $n=24$ and $t_1 p_1$ otherwise. In addition, add the edges $x p_1, \dots, x p_{n-24}, x z_1, x y, x z_3$ for $x \in \{t_2,t_3\}$ where $y=z_2$ for $x=t_2$ and $y=z_4$ for $x=t_3$. The obtained graph $G_{5,n}$ is a planar triangulation with minimum degree 4. 

     \def\GraphUnit{1.4} 
     \begin{figure}[htb]
	\centering
\begin{tikzpicture}[scale=0.45,rotate=-90,x=\GraphUnit cm,
	y=\GraphUnit cm,main_node/.style={circle,draw,minimum size=1em,inner sep=3pt]}]

	\node[main_node] (v1)  at (-5.0, 0.0)   {$t_3$};
	\node[main_node] (v5)  at ( 5.0, 0.0)   {$t_2$};
	\node[main_node] (v6)  at ( 0.0, 8.0)   {$t_1$};
	
	\node[main_node] (v0)  at (-1.0, 0.2909+0.6) {};
	\node[main_node] (v2)  at ( 0.0, 0.7273+0.9) {};
	\node[main_node] (v3)  at ( 1.0, 0.2909+0.6) {};
	\node[main_node] (v4)  at ( 0.0, 0.1455+0.2) {};
	
	\node[main_node] (v7)  at (-2.0+0.5, 3.6364+1.3) {};
	\node[main_node] (v8)  at (-2.0+0.5, 2.4727+1.3) {};
	\node[main_node] (v9)  at ( 0.0, 2.1818+2.3) {};
	\node[main_node] (v10) at (-2.0+0.5, 1.0182+1.5) {};
	\node[main_node] (v11) at (-2.0+0.5, 0.7273+1.1) {};
	\node[main_node] (v12) at (-1.0+0.5, 1.1636+1.5) {};
	\node[main_node] (v13) at ( 1.0-0.5, 1.1636+1.5) {};
	\node[main_node] (v14) at ( 2.0-0.5, 0.7273+1.1) {};
	\node[main_node] (v15) at ( 2.0-0.5, 1.0182+1.5) {};
	\node[main_node] (v16) at ( 2.0-0.5, 2.4727+1.3) {};
	\node[main_node] (v17) at ( 2.0-0.5, 3.6364+1.3) {};
	\node[main_node] (v18) at ( 1.0-0.5, 4.0727+1.3) {};
	\node[main_node] (v19) at (-1.0+0.5, 4.0727+1.3) {};

	\node[main_node,fill=green!60] (pn24) at ( 0.0, 11.0)   {$p_{n-24}$};
	\node[main_node, draw=none] (lc) at ( 0.0, 12)   {};
	\node[main_node, draw=none] (rc) at ( 0.0, 13.5)   {};
	\node[main_node,fill=green!60] (p2) at ( 0.0, 14)   {$p_2$};
	\node[main_node,fill=green!60] (p1) at ( 0.0, 16.5)   {$p_1$};
	
	\node[main_node,fill=green!60] (z1) at ( 0.0, 19)   {$z_1$};
	\node[main_node,fill=green!60] (z3) at ( 0.0, 21.5)   {$z_3$};
	
	\node[main_node,fill=green!60] (z2) at ( 1.0, 20.25)   {$z_2$};
	\node[main_node,fill=green!60] (z4) at ( -1.0, 20.25)   {$z_4$};

	\draw[thick] (v0) -- (v1);
	\draw[thick] (v0) -- (v2);
	\draw[thick] (v0) -- (v3);
	\draw[thick] (v0) -- (v4);
	
	\draw[thick] (v1) -- (v2);
	\draw[thick] (v1) -- (v4);
	\draw[thick] (v1) -- (v5);
	\draw[thick] (v1) -- (v6);
	\draw[thick] (v1) -- (v7);
	\draw[thick] (v1) -- (v8);
	\draw[thick] (v1) -- (v9);
	\draw[thick] (v1) -- (v10);
	\draw[thick] (v1) -- (v11);
	
	\draw[thick] (v2) -- (v3);
	\draw[thick] (v2) -- (v5);
	\draw[thick] (v2) -- (v9);
	\draw[thick] (v2) -- (v11);
	\draw[thick] (v2) -- (v12);
	\draw[thick] (v2) -- (v13);
	\draw[thick] (v2) -- (v14);
	
	\draw[thick] (v3) -- (v4);
	\draw[thick] (v3) -- (v5);
	\draw[thick] (v4) -- (v5);
	
	\draw[thick] (v5) -- (v6);
	\draw[thick] (v5) -- (v9);
	\draw[thick] (v5) -- (v14);
	\draw[thick] (v5) -- (v15);
	\draw[thick] (v5) -- (v16);
	\draw[thick] (v5) -- (v17);
	
	\draw[thick] (v6) -- (v7);
	\draw[thick] (v6) -- (v9);
	\draw[thick] (v6) -- (v17);
	\draw[thick] (v6) -- (v18);
	\draw[thick] (v6) -- (v19);
	
	\draw[thick] (v7) -- (v8);
	\draw[thick] (v7) -- (v19);
	\draw[thick] (v8) -- (v9);
	\draw[thick] (v8) -- (v19);
	
	\draw[thick] (v9) -- (v10);
	\draw[thick] (v9) -- (v12);
	\draw[thick] (v9) -- (v13);
	\draw[thick] (v9) -- (v15);
	\draw[thick] (v9) -- (v16);
	\draw[thick] (v9) -- (v18);
	\draw[thick] (v9) -- (v19);
	
	\draw[thick] (v10) -- (v11);
	\draw[thick] (v10) -- (v12);
	\draw[thick] (v11) -- (v12);
	
	\draw[thick] (v13) -- (v14);
	\draw[thick] (v13) -- (v15);
	\draw[thick] (v14) -- (v15);
	
	\draw[thick] (v16) -- (v17);
	\draw[thick] (v16) -- (v18);
	\draw[thick] (v17) -- (v18);

	\draw[thick] (p1) -- (p2);
	\draw[thick] (z1) -- (p1);
	\draw[thick] (z1) -- (z2);
	\draw[thick] (z1) -- (z4);
	\draw[thick] (z2) -- (z3);
	\draw[thick] (z2) -- (z4);
	\draw[thick] (z3) -- (z4);
	
	 \path[draw, thick]
	(v6) edge[gray, dashed] node {} (pn24)
	(v5) edge[gray, dashed] node {} (pn24)
	(v1) edge[gray, dashed] node {} (pn24)
	(v5) edge[gray, dashed] node {} (p2)
	(v1) edge[gray, dashed] node {} (p2)
	(v5) edge[gray, dashed] node {} (p1)
	(v1) edge[gray, dashed] node {} (p1)
	(v5) edge[gray, dashed] node {} (z1)
	(v1) edge[gray, dashed] node {} (z1)
	(v5) edge[gray, dashed] node {} (z2)
	(v1) edge[gray, dashed] node {} (z4)
	;

	\path[draw, thick, dotted]
	(lc) edge node {} (rc)
	;

    \draw[thick, rounded corners=10pt, gray, dashed]
  (v1) .. controls ($(z3)+(-3,0)$) ..  (z3);

  \draw[thick, rounded corners=10pt, gray, dashed]
  (v5) .. controls ($(z3)+(3,0)$) ..  (z3);

\end{tikzpicture}

\caption{On the left-hand side, a planar 20-vertex triangulation $\baseMinDegFour$ with 384 perfect matchings. 
On the right-hand side, the graph $R_m$ with green vertices. The dashed edges connect $\baseMinDegFour$ to $R_{n-20}$, making the entire depicted graph $G_{5,n}$ a planar triangulation with minimum degree 4 and 768 perfect matchings.
}
\label{fig:minDeg4}
\end{figure}
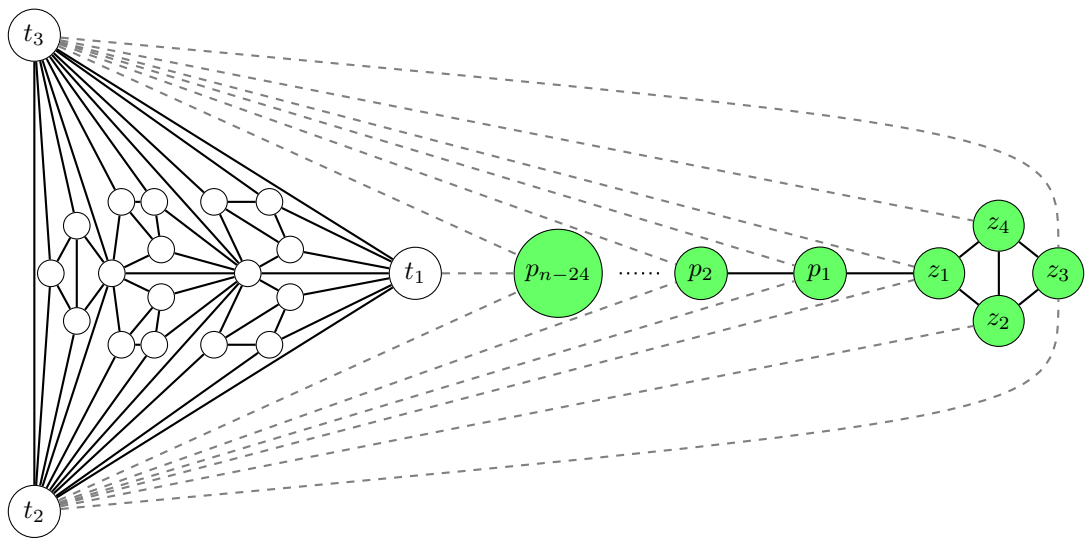

    By applying \cref{lem:glue} with $H=\baseMinDegFour$, $D=\{t_1,t_2,t_3\}$, and $L=R_{n-20}$, we conclude that $\pm(G_{5,n})=\pm(\baseMinDegFour) \cdot \pm(R_{n-20})=768$. \end{proof}

The following natural question remains open. Is there a constant $c < 12$ such that there are infinitely many matchable planar 3-connected graphs, each with exactly $c$ perfect matchings? We note that such graphs must have at least six perfect matchings \cite{L72}.

\section{At least a linear number of perfect matchings:\\
The 4- and 5-connected cases}\label{sec:4and5connPl}

\subsection{The 4-connected case}

An $n$-vertex graph is called \textit{$k$-extendable} if any matching of size $k<n/2$ can be extended to a perfect matching. Every planar $4$-connected graph of even order is 1-extendable, see~\cite{P92}. This also follows from the theorem that in a planar 4-connected graph, every edge lies in a Hamiltonian cycle, see for instance \cite{OV14}. 
This together with Lemma~\ref{lem:matchspacedim} immediately yields the following.

\begin{corollary}
Every planar $4$-connected graph of even order contains $\Omega(n)$ perfect matchings.
\end{corollary}

The \textit{double wheel graph} $DW_n$ on $n$ vertices is the planar 4-connected graph defined as the join of an $(n-2)$-cycle and $2K_1$. The following fact is obvious.

\begin{proposition}
    If $n$ is even, then $\pm(DW_n)=\frac{(n-2)^2}{2}$.
\end{proposition}

There are special cases when we can guarantee an exponential number of perfect matchings. Two cases will be discussed in \cref{prop:5connExp,prop:4connExp}, respectively, and 
two other ones are as follows: Let $G$ be a planar 4-connected graph on $n$ vertices. If (1) $G$ contains a $3$-factor $F$ such that all connected components of $F$ are matchable and the order of the union of the bridgeless connected components of $F$ is linear in $n$; or (2) $G$ is 4-regular, then by \cite{CS12} and \cite{Esperet2011}, respectively, $G$ must contain at least an exponential number of perfect matchings.

\subsection{The 5-connected case}

We first show that every edge in a planar 5-connected graph is contained in a linear number of perfect matchings, something that might be true in planar 4-connected graphs, but we are not able to prove this. Thereafter, we show that it follows from a result of Aldred, Alahmadi, and Thomassen  that planar 5-connected triangulations contain an exponential number of perfect matchings.

\begin{proposition}\label{prop:5connPlEdgeInLinPMs}
    Let $G$ be a planar $5$-connected graph of even order $n$. Then every edge in $G$ is contained in $\Omega(n)$ perfect matchings.
\end{proposition}

\begin{proof}

    Take an arbitrary edge $e=uv$ in a planar 5-connected $n$-vertex graph $G$, with $n$ even. Consider the graph $H = G-u-v$, which is 3-connected and 1-extendable since $G$ is 2-extendable~\cite{P92}. These two properties of $H$ imply that $|E(F(H))|=|E(H)|\ge \frac{3(n-2)}{2}$ and hence applying \cref{lem:matchspacedim} on $H$ gives $\pm(H) \ge\frac{n+6}{4}$. We can extend each perfect matching $M$ of $H$ to a perfect matching in $G$ by adding the edge $e$ to $M$. Thus, $e$ is contained in at least $\frac{n+6}{4}$ perfect matchings.
\end{proof}

We are able to give a full answer for the triangulation case. Alahmadi, Aldred, and Thomassen~\cite{AAT20} proved that every 5-connected planar or projective planar triangulation of even order $n$ contains at least $\alpha^n$ Hamiltonian cycles, for some real $\alpha > 1$. It is not difficult to adapt this proof to show that the same statement holds if in the preceding sentence ``Hamiltonian cycles'' is replaced by ``perfect matchings''. But there is a more direct way. Consider a graph $G$ of even order and let $h$ be the number of Hamiltonian cycles in $G$. Since every Hamiltonian cycle in $G$ is of the form $M_1 \cup M_2$ for suitable perfect matchings $M_1, M_2$, we have $h \le \pm(G)^2.$ The next statement now follows directly from the aforementioned theorem of Alahmadi, Aldred, and Thomassen.

\begin{proposition}\label{prop:5connExp}
    
    Every $5$-connected planar or projective planar triangulation of even order $n$ contains at least $\beta^n$ perfect matchings, for some real $\beta > 1$.
    
\end{proposition}

In fact, one can show something somewhat stronger than \cref{prop:5connExp}. In~\cite{LQ22}, Lo and Qian adapt the above approach of Alahmadi, Aldred, and Thomassen to show the following result if one replaces ``perfect matchings'' by ``Hamiltonian cycles''. By the relationship between the number of Hamiltonian cycles and the number of perfect matchings we mentioned earlier, one immediately obtains the following result.

\begin{proposition}\label{prop:4connExp}
    Every $4$-connected planar or projective planar triangulation of even order $n$ and $O(n)$ $4$-separators has exponentially many perfect matchings.
\end{proposition}

\section{Counterexamples to a conjecture of Gr\"unbaum and Zaks}\label{sec:conjZaks}

For this last section we drop the planarity constraint and provide a structurally simple set of counterexamples to a published conjecture on the number of perfect matchings in $k$-connected graphs. More specifically, in \cite{Z71}, Zaks formulates a number of conjectures; he states that ``[m]ost of the conjectures here are due to Gr\"unbaum.'' 

We focus here on his Conjecture~5.

\begin{conjecture}[Conjecture~5 in~\cite{Z71}]
    Consider integers $k \ge 3$ and $n \ge 2k$. Then every matchable $k$-connected $n$-vertex graph $G$ satisfies $\pm(G) \ge k!$; and there are infinitely many graphs realising this.
\end{conjecture}

The last statement is certainly true, as Zaks describes $k$-connected $n$-vertex graphs with $\pm(G) = k!$, for any even $n \ge 2k$. However, for infinitely many $k$ there are $k$-connected graphs with $2k$ vertices with significantly fewer than $k!$ perfect matchings. These constitute counterexamples to Conjecture~5.

\begin{proposition}
    For any $\varepsilon > 0$ there is a $k_0$ so that for any $k\ge k_0$ we have $\pm(K_k \square K_2) < \varepsilon k!$. 
\end{proposition}
\begin{proof}

    Consider the prism over a complete graph, that is: the Cartesian product $\Pi_k := K_k \square K_2$. We denote by $M$ the set of edges of $\Pi_k$ not contained in one of the two copies of $K_k$. 

    A perfect matching in $\Pi_k$ uses some number $i$ of edges in $M$, $0 \le i \le |M|$, such that the remaining number of vertices $k-i$ in each $K_k$ is even. Since the number of perfect matchings in a complete graph of order $2j$ is $(2j-1)!!$, we have

$$\pm(\Pi_k) = \sum_{j=0}^{\lfloor k/2\rfloor} \binom{k}{2j}\bigl((2j-1)!!\bigr)^2.$$

Put $a_k := \pm(\Pi_k)/k!$. By

 \[
(2j-1)!!=\frac{(2j)!}{2^j j!}
 \]
we get
\[
a_k
=
\sum_{j=0}^{\lfloor k/2\rfloor}
\frac{\binom{2j}{j}}{4^j\,(k-2j)!}.
\]

\noindent We now show that $\lim_{k\to\infty} a_k = 0.$
Let
\[
b_j:=\frac{\binom{2j}{j}}{4^j}.
\]

We have \(0\le b_j\le 1\) for all \(j\), since \(\binom{2j}{j}\le 4^j\). Furthermore, it can be inferred from Stirling's work that 
$n! = \sqrt{2\pi n}\,(n/e)^n(1+O(\frac{1}{n}))$. This yields
$$\binom{2j}{j} = \frac{(2j)!}{(j!)^2} \sim \frac{\sqrt{4\pi j}\,(2j/e)^{2j}}{(\sqrt{2\pi j}\,(j/e)^j)^2} = \frac{\sqrt{4\pi j}\,2^{2j} j^{2j}/e^{2j}}{2\pi j \cdot j^{2j}/e^{2j}} = \frac{4^j}{\sqrt{\pi j}}$$
and
$$\frac{\binom{2j}{j}}{4^j} = \frac{1}{\sqrt{\pi j}}\left(1 + O\!\left(\tfrac{1}{j}\right)\right).$$
Hence, there is a constant \(C>0\) such that
\[
b_j\le \frac{C}{\sqrt{j+1}}
\qquad\text{for all } j\ge 0.
\]
Now split the sum defining \(a_k\) into two parts:
\[
a_k = \sum_{j=0}^{\lfloor k/2\rfloor}\frac{b_j}{(k-2j)!} = \sum_{j=0}^{\lfloor k/4\rfloor}\frac{b_j}{(k-2j)!}
\;+\;
\sum_{j=\lfloor k/4\rfloor+1}^{\lfloor k/2\rfloor}\frac{b_j}{(k-2j)!}.
\]
Let us denote the first sum by $S_1$ and the second sum by $S_2$. For \(S_1\), we have \(k-2j\ge \lfloor k/2 \rfloor \), hence
\[
S_1 \le \sum_{j=0}^{\lfloor k/4\rfloor}\frac{1}{\lfloor k/2 \rfloor !}
\le \frac{k/4+1}{\lfloor k/2\rfloor!}\xrightarrow[k\to\infty]{}0.
\]
For \(S_2\), we have \(j>\frac{k}{4}\), so \(j+1\ge \frac{k}{4}\), and therefore
\[
b_j\le \frac{C}{\sqrt{j+1}}\le \frac{2C}{\sqrt{k}}.
\]
Thus
\[
S_2
\le \frac{2C}{\sqrt{k}}\sum_{j=0}^{\lfloor k/2\rfloor}\frac{1}{(k-2j)!}
\le \frac{2C}{\sqrt{k}}\sum_{m=0}^{\infty}\frac{1}{m!}
= \frac{2Ce}{\sqrt{k}}
\xrightarrow[k\to\infty]{}0.
\]

\end{proof}

\section{Notes}\label{sec:notes}

We conclude this article with some notes, open problems, and possible directions for future research.

\bigskip

\noindent \textbf{1.} Consider integers $k \ge 3$ and $n \ge 2k$, and a matchable $k$-connected $n$-vertex graph $G$. By the result presented in Section~4, we know that $G$ does not necessarily satisfy $\pm(G) \ge k!$. But what would be a non-trivial lower bound?

\bigskip

\noindent \textbf{2.} We have infinite families of planar 3-connected graphs with a non-zero constant number of perfect matchings for girth 3 and girth 4, but not for girth 5. 

This naturally raises the following question.

\begin{problem}
Does there exist an infinite family of matchable planar $3$-connected graphs of girth $5$ with a constant number of perfect matchings? 
\end{problem}

We verified computationally that all matchable planar 3-connected graphs of girth 5 and even order at most 50 are matching covered. We used \texttt{plantri} \cite{BM07} for the generation of planar graphs and our own code (\href{https://github.com/AGT-Kulak/countpm}{https://github.com/AGT-Kulak/countpm}) for checking whether a graph is matching covered.

We also do not know whether there exist infinitely many planar 3-connected graphs with minimum degree 5 which have a constant number of perfect matchings.

\bigskip

\noindent \textbf{3.} We conjecture the following.

\begin{conjecture}
    For large orders, the double wheels have the minimum number of perfect matchings among all planar $4$-connected triangulations.
\end{conjecture}

This mirrors the Hakimi-Schmeichel-Thomassen Conjecture stating that double wheels have the minimum number of Hamiltonian cycles among all planar 4-connected triangulations \cite{HST79}. Liu, Wang, and Yu~\cite{LWY22} showed that planar 4-connected triangulations have at least quadratically many Hamiltonian cycles; note that double wheels have a quadratic number of Hamiltonian cycles.

\subsection*{Acknowledgements}

\noindent  
 We would like to thank Nishad Kothari, Davide Mattiolo, Brendan McKay, and Sreejith K.~Pallathumadam for interesting discussions on the contents of this manuscript.

The research of Jan Goedgebeur and Tibo Van den Eede was supported by Internal Funds of KU Leuven and a grant of the Research Foundation Flanders (FWO) with grant number G0AGX24N. 
Moreover, Tibo Van den Eede was also supported by an FWO travel grant with grant number V401626N. Jorik Jooken was supported by an FWO Postdoctoral Fellowship with grant number 1222524N.

\bibliographystyle{graphtheory}
\bibliography{ref}

@article{Esperet2011,
title = {Exponentially many perfect matchings in cubic graphs},
journal = {Adv. Math.},
volume = {227},
number = {4},
pages = {1646--1664},
year = {2011},
issn = {0001-8708},
doi = {https://doi.org/10.1016/j.aim.2011.03.015},
author = {L. Esperet and F. Kardo{\v{s}}  and A. D. King and D. Kr{\'a}l' and S. Norine},
}

@article{CS12,
title = {Perfect matchings in planar cubic graphs},
journal = {Combinatorica},
volume = {32},
number = {4},
pages = {403--424},
year = {2012},
doi = {http://dx.doi.org/10.1007/s00493-012-2660-9},
author = {M. Chudnovsky and P. Seymour},
}

@article{EPL82,
  title={Brick decompositions and the matching rank of graphs},
  author={Edmonds, J. and Pulleyblank, W. R. and Lov{\'a}sz, L.},
  journal={Combinatorica},
  volume={2},
  number={3},
  pages={247--274},
  year={1982},
  publisher={Springer}
}

@article{LWY22,
  title={Counting Hamiltonian cycles in planar triangulations},
  author={Liu, X. and Wang, Z. and Yu, X.},
  journal={J. Combin. Theory Ser. B},
  volume={155},
  pages={256--277},
  year={2022},
  publisher={Elsevier}
}

@article{P92,
	title={Extending matchings in planar graphs {IV}},
	author={Plummer, M. D.},
	journal={Discrete Math.},
	volume={109},
	number={1-3},
	pages={207--219},
	year={1992},
	publisher={Elsevier}
}

@article{Z71,
	title={On the 1-factors of $n$-connected graphs},
	author={Zaks, J.},
	journal={J. Combin. Theory Ser. B},
	volume={11},
	number={2},
	pages={169--180},
	year={1971},
	publisher={Elsevier}
}

@article{L72,
	title={On the structure of factorizable graphs},
	author={Lov{\'a}sz, L.},
	journal={Acta Math. Hung.},
	volume={23},
	number={1-2},
	pages={179--195},
	year={1972},
	publisher={Akad{\'e}miai Kiad{\'o}, co-published with Springer Science+ Business Media BV~…}
}

@article{AAT20,
	title={Cycles in 5-connected triangulations},
	author={Alahmadi, A. and Aldred, R. E. L. and Thomassen, C.},
	journal={J.~Combin. Theory Ser.~B},
	volume={140},
	pages={27--44},
	year={2020},
	publisher={Elsevier}
}

@article{M76,
  title={{\"U}ber die {A}nzahl der 1-{F}aktoren in 2-fach zusammenh{\"a}ngenden {G}raphen},
  author={Mader, W.},
  journal={Math. Nachr.},
  volume={74},
  number={1},
  pages={217--232},
  year={1976},
  publisher={Wiley Online Library}
}

@article{HST79,
  title={On the number of Hamiltonian cycles in a maximal planar graph},
  author={Hakimi, S. L. and Schmeichel, E. F. and Thomassen, C.},
  journal={J. Graph Theory},
  volume={3},
  number={4},
  pages={365--370},
  year={1979},
  publisher={Wiley Online Library}
}

@book{YL09,
  title={Graph factors and matching extensions},
  author={Yu, Q. R. and Liu, G.},
  year={2009},
  publisher={Springer}
}

@article{LQ22,
  title={Hamiltonian cycles in 4-connected planar and projective planar triangulations with few 4-separators},
  author={Lo, O. S. and Qian, J.},
  journal={SIAM J. Discrete Math.},
  volume={36},
  number={2},
  pages={1496--1501},
  year={2022},
  publisher={SIAM}
}

@article{OV14,
  title={2-edge-Hamiltonian-connectedness of 4-connected plane graphs},
  author={Ozeki, K. and Vr{\'a}na, P.},
  journal={European J. Combin.},
  volume={35},
  pages={432--448},
  year={2014},
  publisher={Elsevier}
}

@article{BM07,
  title={Fast generation of planar graphs},
  author={Brinkmann, G. and McKay, B. D.},
  journal={MATCH Commun. Math. Comput. Chem.},
  volume={58},
  number={1},
  pages={323--357},
  year={2007}
}

@article{HoG,
  title={House of Graphs 2.0: A database of interesting graphs and more},
  author={Coolsaet, K. and D’hondt, S. and Goedgebeur, J.},
  journal={Discrete Appl. Math.},
  volume={325},
  pages={97--107},
  year={2023},
  publisher={Elsevier},
  note = {{A}vailable at: \url{https://houseofgraphs.org}}  
}

\end{document}